\documentclass[11pt]{amsart}
\usepackage{amsmath, amssymb, amscd, amsthm, amsfonts, latexsym}
\usepackage{graphicx}
\usepackage{tikz-cd}
\usepackage{hyperref}
\usepackage{mathtools}
\usepackage{mathrsfs}
\usepackage{enumerate}
\usepackage{mathabx}
\usepackage{setspace}
\usepackage{color}
\usepackage{comment}
\usepackage{tikz}
\usepackage{tikz-cd}

\theoremstyle{plain}
\newtheorem{thm}{Theorem}[section]
\newtheorem{prop}[thm]{Proposition}
\newtheorem{lemm}[thm]{Lemma}
\newtheorem{cor}[thm]{Corollary}

 \newcommand{\Wi}{\widetilde}

\theoremstyle{definition}
\newtheorem{defn}[thm]{Definition}

\newtheorem{rmk}[thm]{Remark}
\newtheorem{con}[thm]{Conjecture}

\def\Z{{\mathbb Z}}

\def\N{{\mathbb N}}
\def\S{{\mathbb S}}
\def\cd{\protect\operatorname{cd}}
\def\cat{\protect\operatorname{cat}}
\def\Ker{\protect\operatorname{Ker}}
\def\nil{\protect\operatorname{nil}}

\def\Coind{\protect\operatorname{Coind}}

\title{On the LS-category of homomorphisms}
\author{Alexander Dranishnikov and Nursultan Kuanyshov}

\address{Alexander  Dranishnikov, Department of Mathematics, University of Florida, 358 Little Hall, Gainesville, FL 32611-8105, USA}
\email{dranish@math.ufl.edu}

\address{Nursultan Kuanyshov, Department of Mathematics, University of Florida, 358 Little Hall, Gainesville, FL 32611-8105, USA}
\email{kuanyshov@math.ufl.edu}



\begin{document}
\maketitle
\begin{abstract}
We prove the equality $\cat(\phi)=\cd(\phi)$ for homomorphisms $\phi:\Gamma\to \Lambda$ of a  torsion free finitely generated nilpotent groups $\Gamma$ to an arbitrary group $\Lambda$.
We construct an epimorphism $\psi:G\to H$ between geometrically finite groups with $\cat(\psi)> \cd(\psi)$.
\end{abstract}

\section{Introduction}
\begin{defn}
The (reduced) Lusternik-Schnirelmann category (LS-category), $\cat X$, of an ANR space $X$ is the minimal number $k$ such that $X$ admits an open cover by $k+1$ sets $U_0,U_1,\dots,U_k$ such that each $U_i$ is contractible in $X$.
\end{defn}
The Lusternik-Schnirelmann category is an important invariant, since it gives a lower bound on the number of critical points for a smooth real-valued function on a closed manifold~\cite{LS}.
Since it is a homotopy invariant, it can be defined for discrete groups $\Gamma$ as $\cat\Gamma=\cat B\Gamma$ where $B\Gamma=K(\Gamma, 1)$ is a classifying space.
Eilenberg and Ganea~\cite{EG} proved that the LS-category of a discrete group equals its cohomological dimension,
$\cat(\pi)=\cd(\pi)$. We recall that
the cohomological dimension of a group $\Gamma$  is defined as follows,
 $$\cd(\Gamma)=\max\{k\mid H^k(\Gamma,M)\ne 0\}$$ where the maximum is taken over all $\Z\Gamma$-modules $M$. 
\begin{thm}[\cite{Sch,DR}]
For the cohomological dimension of a discrete group $\Gamma$,
$$\cd(\Gamma)=\max\{k\mid(\beta_\Gamma)^k\ne 0\}$$ where $\beta_\Gamma\in H^1(\Gamma,I(\Gamma))$ is the Berstein-Schwarz class of $\Gamma$.
\end{thm}

The {\em LS-category of the map} $f:X\to Y$, $\cat f$, is the minimal number $k$ such that X admits an open cover by $k+1$ open sets $U_{0},U_{2},..., U_{k}$ with  nullhomotopic restrictions $f|_{U_i}:U_i\to Y$ for all $i$. The LS-category $\cat\phi$ of a group homomorphism $\phi:\Gamma\to\pi$ is defined as $\cat f$ where the map
$f:B\Gamma\to B\pi$ induces the homomorphism $\phi$ for the fundamental groups. 

The {\em cohomological dimension} $\cd(\phi)$ of a group homomorphism $\phi:\Gamma\to\pi$ was defined by Mark Grant about 10 years ago on Mathoverflow~\cite{Gr} as maximum of $k$ such that
there is a $\pi$-module $M$ with  the nonzero  induced  homomorphism $\phi^*:H^k(\pi,M)\to H^k(\Gamma,M)$.
In view of universality of the Berstein-Schwarz class~\cite{DR} for any homomorphism $\phi:\Gamma\to\pi$
$$
\cd(\phi)=\max\{k\mid \phi^*(\beta_\pi)^k\ne 0\}.
$$
This brings the inequality $\cd(\phi)\le\cat\phi$ for all homomorphisms.

In view of the eqality $\cd(\Gamma)=\cat\Gamma$, the following conjecture seems to be natural:
\begin{con}\label{con}
For any group homomorphism $\phi:\Gamma\to\pi$ always
$$
\cat \phi=\cd(\phi).
$$
\end{con}
In~\cite{Sc} Jamie Scott considered this conjecture for geometrically finite groups and he proved it for monomorphisms of any groups
and for homomorphisms of free and free abelian groups. 
 In ~\cite{Gr} Tom Goodwillie gave an example of an epimorphism of an  infinitely generated group $\phi:G\to\Z^2$ with $cd(\phi)=1$ that  disproves the conjecture.

In the first part of this paper we prove Conjecture~\ref{con} for finitely generated torsion free nilpotent groups. In the second part we present a 
finitely generated
counterexample  by constructing a 
map between aspherical manifolds $f:M\to N$ with $$\cat f>\cd(f_\#:\pi_1(M)\to\pi_1(N)).$$

\section{Preliminaries}

\subsection{Nilpotent groups}
The upper central series of a group $\Gamma$ is a chain of subgroups 
$${e}=Z_{0} \leqslant Z_{1} \leqslant....\leqslant Z_{n} \leqslant.... $$
where $Z_{1}=Z(\Gamma)$ is the center of the group, and $Z_{i+1}$ is the preimage under the canonical epimorphism $\Gamma\to\Gamma/Z_{i}$ of the center of $\Gamma/Z_{i}$.
A group $\Gamma$ is {\em nilpotent} if $Z_{n}=\Gamma$ for some $n$.
The least such $n$ is called  {\em the nilpotency class} of $\Gamma$, denoted $\nil(\Gamma)$. Note that the groups with the nilpotency class one are exactly abelian groups.

The lower central series of a group $\Gamma$ is a chain of subgroups 
$$\Gamma=\gamma_0(\Gamma)\ge \gamma_1(\Gamma)\ge\gamma_2(\Gamma)\ge...$$
defined as $\gamma_{i+1}(\Gamma)=[\gamma_i(\Gamma),\Gamma]$.
It's known that for nilpotent groups $\Gamma$ the nilpotency class $\nil(\Gamma)$ equals the least $n$ for which $\gamma_n(\Gamma)=1$.

\begin{prop}\label{nil}
(1) Let $\phi:\Gamma\rightarrow \Gamma'$ be an epimorphism. Then $\phi(Z(\Gamma))\subset Z(\Gamma')$ and $\phi(\gamma_i(\Gamma))=\gamma_i(\Gamma')$ for  all $i$.

(2) For any finitely generated torsion free nilpotent group $\Gamma$, any $z\in\Gamma$, and any $n\in\N$ the condition $z^n\in Z(\Gamma)$ implies $z\in Z(\Gamma)$.
\end{prop}
\begin{proof}
(1) Straightforward (see for example~\cite{B}, Theorem 5.1.3).

(2) This is Mal'cev's Theorem 1 in~\cite{Ma2}. We note that the proof there is not selfcontained.  This statement
follows from the fact that  $z$ and $z^n$ have the
same centralizers. The latter can be proven using Mal’cev theorem about embedding $\Gamma$  into
the group of unipotent upper triangular matrices ~\cite{Ra} and the fact that $z$ belongs to Zariski
closure of $z^n$.
 \end{proof}

\begin{cor}\label{nil2}
 For any finitely generated torsion free nilpotent group $\Gamma$ the group $\Gamma/Z(\Gamma)$ is torsion free finitely generated nilpotent group.
\end{cor}
\begin{proof}
The torsion free part follows from (2). The rest follows from (1).
\end{proof}

 We note that $\nil(\Gamma/Z(\Gamma)) <\nil(\Gamma)$.
    
Suppose  that $G$ is a connected, simply connected nilpotent Lie group and $\Gamma\subset G$ is a uniform lattice. Then $G$ is the universal cover for $\Gamma$ and $N=G/\Gamma$ is an aspherical manifold, called a {\em nilmanifold}. 

\textbf{Mal'cev Theorem.}~\cite{Ma1} {\em Every torsion free finitely generated nilpotent group $\Gamma$ can be realized as the fundamental group of some nilmanifold. } 

The corresponding simply connected Lie groups $G$ is obtained as the Mal'cev completion of $\Gamma$.

\

\subsection{Ganea-Schwarz's approach to cat} Recall that an element of an iterated join $X_0*X_1*\cdots*X_n$ of topological spaces is a formal linear combination $t_0x_0+\cdots +t_nx_n$ of points $x_i\in X_i$ with $\sum t_i=1$, $t_i\ge 0$, in which all terms of the form $0x_i$ are dropped. Given fibrations $f_i:X_i\to Y$ for $i=0, ..., n$, the fiberwise join of spaces $X_0, ..., X_n$ is defined to be the space
\[
    X_0*_YX_1*_Y\cdots *_YX_n=\{\ t_0x_0+\cdots +t_nx_n\in X_0*\cdots *X_n\ |\ f_0(x_0)=\cdots =f_n(x_n)\ \}.
\]
The fiberwise join of fibrations $f_0, ..., f_n$ is the fibration 
\[
    f_0*_Y*\cdots *_Yf_n: X_0*_YX_1*_Y\cdots *_YX_n \longrightarrow Y
\]
defined by taking a point $t_0x_0+\cdots +t_nx_n$ to $f_i(x_i)$ for any $i$ such that $t_i \neq 0$. 

When $X_i=X$ and $f_i=f:X\to Y$ for all $i$  the fiberwise join of spaces is denoted by $*^{n+1}_YX$ and the fiberwise join of fibrations is denoted by $*_Y^{n+1}f$. 

For a path connected space $X$, we turn an inclusion of a point $*\to X$ into a fibration $p^X_0:G_0(X)\to X$, whose fiber is known to be the loop space $\Omega X$.  The $n$-th Ganea space of $X$ is defined to be the space $G_n(X)=*_X^{n+1}G_0(X)$, while the $n$-th Ganea fibration $p^X_n:G_n(X)\to X$
is the fiberwise join $\ast^{n+1}_Xp^X_0$. Then the fiber of
$p^X_n$ is $\ast^{n+1}\Omega X$.

The following theorem give the Ganea-Shwarz characterization of $\cat$~\cite{Sch},\cite{CLOT}: 
\begin{thm}
If $X$ is a connected ANR, then  $\cat X\le n$ if and only if the fibration $p^X_n:G_n(X)\to X$ admits a section.
\end{thm}
This characterization can be extendend to maps:
\begin{thm}\label{cat map}
If $f:X\to Y$ is a map between connected ANRs, then  $\cat f\le n$ if and only if there is a lift of $f$ with respect to
the fibration $p^Y_n:G_n(Y)\to Y$ admits a section.
\end{thm}

\subsection{Berstein-Schwarz cohomology class}
The Berstein-Schwarz class of a discrete group $\pi$ is the first obstruction $\beta_{\pi}$ to a lift of $B\pi=K(\pi,1)$ to the universal covering $E\pi$. 
Note that $\beta_{\pi}\in H^1(\pi,I(\pi))$ where $I(\pi)$ is the augmentation ideal of the group ring $\Z\pi$~\cite{Be},\cite{Sch}.

\begin{thm}[Universality~\cite{DR},\cite{Sch}]\label{universal}
For any cohomology class $\alpha\in H^k(\pi,L)$ there is a homomorphism of $\pi$-modules $I(\pi)^k\to L$ such that the induced homomorphism for cohomology takes $(\beta_{\pi})^k\in H^k(\pi;I(\pi)^k)$ to $\alpha$  where $I(\pi)^k=I(\pi)\otimes\dots\otimes I(\pi)$ and $(\beta_{\pi})^k=\beta_{\pi}\smile\dots\smile\beta_{\pi}$.
\end{thm}

In the paper we use notations $H^*(\Gamma,A)$ for cohomology of a group $\Gamma$ with coefficient in $\Gamma$-module $A$. The cohomology groups of a space $X$ with the fundamental group $\Gamma$ we denote as $H^*(X;A)$. Thus, $H^*(\Gamma,A)=H^*(B\Gamma;A)$ where $B\Gamma=K(\Gamma,1)$.

\section{Reduction to epimorphisms}

\begin{lemm}\label{lift}
Let $\pi\subset\Lambda$ be a subgroup,  $j:B\pi\to B\Lambda$ be a map generated by this inclusion, and $p':j^*E\Lambda\to B\pi$ be the pull-back of the universal covering $p_\Lambda:E\Lambda\to
B\Lambda$. Then $p'$ has a lift with respect to the universal covering $p_\pi:E\pi\to B\pi$.
\end{lemm}
\begin{proof}
Clearly, for each path component $C$ of $j^*E\Lambda$  the restriction
$p'|_C:C\to B\pi$  of a covering $p'$ is a covering. Easy diagram chasing shows that $C$
 is simply connected. Hence, each path component $C$ of $j^*E\Lambda$ is homeomorphic to $E\pi$ and the restriction
$p'|_C:C\to B\pi$ is the universal covering. We may assume that $j^*E\Lambda$ is a CW complex. This would imply that each path component $C$ is open.
Thus, the lift $j^*E\Lambda\to E\pi$ can be defined independently on each path component.
\end{proof}

Given a homomorphism $\phi:\Gamma\to\Lambda$, by $\phi':\Gamma\to im(\phi)$ we denote the restriction of $\phi$ from the codomain to its range. The following theorem
is a generalization of Theorem 6.11 from~\cite{Sc}.
\begin{thm}\label{cat of mono}
For any group homomorphism $\phi:\Gamma\to\Lambda$, $\cat\phi=\cat\phi'$ .
\end{thm}
\begin{proof}
Clearly, $\cat\phi\le\cat\phi'$. We show that  $\cat\phi\ge\cat\phi'$. Let $\cat\phi=k$. This means that there is a lift of the corresponding map $f:B\Gamma\to B\Lambda$ with respect to the Ganea fibration $p_k:G_k(B\Lambda)\to B\Lambda$.  Since the path fibration $p:P_0(B\pi)\to B\pi$ is fiber-wise homotopy equivalent to the universal covering $E\pi\to B\pi$ for any discrete group $\pi$, the $k$-th Ganea fibration $G_k(B\pi)\to B\pi$ is fiberwise homotopy equivalent to the  iterated  fiberwise join $*^{k+1}_{B\pi}E\pi$ of the universal covering.
Then $f$ admits a lift with respect to 
$$*^{k+1}_{B\Lambda}p_\Lambda:*^{k+1}_{B\Lambda}E\Lambda\to B\Lambda.$$ 
The map $f$ factors as $f=j\circ f'$ where $f':B\Gamma\to B\pi$, $j:B\pi\to B\Lambda$ and
$\pi=im(\phi)$. Hence, the map $f'$ admits a lift to the pull-back $$\bar f':B\Gamma\to j^*(*^{k+1}_{B\Lambda}E\Lambda)$$
By Lemma~\ref{lift} there is a lift $s$ of $*^{k+1}_{B\pi}p':j^*(*^{k+1}_{B\Lambda}E\Lambda)\to B\pi$ with respect to $$*^{k+1}_{B\pi}p_\pi:*^{k+1}_{B\pi}E\pi\to B\pi.$$
Then the composition $s\circ \bar f'$ is a lift of $f'$ with respect to $*^{k+1}_{B\pi}p_\pi$. By Theorem~\ref{cat map} $\cat f'\le k$ and, hence, $\cat\phi'\le k$.
\end{proof}
\begin{thm}\label{cd of mono}
For any group homomorphism $\phi:\Gamma\to\Lambda$ and $\phi':\Gamma\to im(\phi)=\pi$ as above
$$\cd(\phi)=\cd(\phi').$$
\end{thm}
\begin{proof}
Clearly, $\cd(\phi')\ge\cd(\phi)$. We show that $\cd(\phi')\le\cd(\phi)$.  Let $\cd(\phi')=k$. Then $\phi'^*:H^k(\pi,M)\to H^k(\Gamma,M)$ is not zero for some $\pi$-module $M$.
Let $$\alpha:\Coind^{\Lambda}_\pi M=Hom_\pi(\Z\Lambda,M)\to M$$ denote the canonical $\pi$-homomorphism, defined for $f:\Lambda\to M$ as $\alpha(f)=f(1)$.
Consider the commutative diagram
$$
\begin{CD}
H^k(\Lambda, \Coind^{\Lambda}_{\pi} M) @>j^*>>H^k(\pi,\Coind^{\Lambda}_{\pi} M) @>\alpha^*>> H^k(\pi,M)\\
@. @V\phi'^*VV @V\phi'^*VV\\
@. H^k(\Gamma,\Coind^{\Lambda}_\pi M) @>\alpha^*>> H^k(\Gamma,M)\\
\end{CD}
$$
where $\alpha^*$ is the coefficient homomorphism generated by $\alpha$.
By Shapiro Lemma~\cite{Br} the top row through homomorphism is an isomorphism. Therefore, the homomorphism
$$
\phi^*=\phi'^*j^*:H^k(\Lambda, \Coind^\Lambda_\pi M)\to H^k(\Gamma,\Coind^\Lambda_\pi M)$$ is nonzero. Hence, $\cd(\phi)\ge k$.
\end{proof}
Theorem~\ref{cat of mono} and Theorem~\ref{cd of mono} imply the following:
\begin{cor}
Suppose that Conjecture~\ref{con} holds true for all epimorphisms $\phi:\Gamma\to\pi$ for some class of groups. Then it holds for all homomorphisms for groups from that class.
\end{cor}

\section{Homomorphisms of nilpotent groups}

\begin{lemm}\label{nilbundle}
 Let $\Gamma$ be $\pi$ are finitely generated, torsion free nilpotent groups. Then every epimorphism $\phi:\Gamma\rightarrow \pi$ can be realized as a locally trivial bundle of nilmanifolds with the fiber a nilmanifold.
\end{lemm}
\begin{proof}
We prove it by induction on $n=s+t$ where $s=\nil(\Gamma)$ and $t=\nil(\pi)$ are the nilpotency classes of $\Gamma$ and $\pi$.
The base of induction is the case of abelian groups where any epimorphism $\phi:\Z^{k+\ell}\to\Z^k$ is  the projection onto a factor, since it is a split surjection. Clearly, the projection $\phi$ can be realized as
a trivial fiber bundle of tori $T^{k+\ell}\to T^k$ with the fiber a torus $T^\ell$.

Suppose that $n>2$ and the statement of the lemma holds true for $s+t<n$. We denote by  $B= Z(\Gamma)\cap \Ker\phi$ and by $A=Z(\Gamma)/B$, where $Z(\Gamma)$ is the center of $\Gamma$. Note that $B$ is a direct summand in $Z(\Gamma)$ or, equivalently, $A$ is free. For the later we claim that if $z^n\in B$, then $z\in B$. Indeed,  If $z^n\in\Ker\phi$ then $z\in\Ker\phi$, since $\pi$ is torsion free. If $z^n\in Z(\Gamma)$ then $z\in Z(\Gamma)$ by 
Proposition~\ref{nil}. Thus, $Z(\Gamma)\cong B\oplus A$. Let $\bar A$ be the direct summand of $Z(\pi)$ that contains $\phi(Z(\Gamma))\cong A$ as a finite index subgroup. 
In view of Corollary~\ref{nil2}, $\pi/\bar A$ is a torsion free nilpotent group. 

We consider the nilpotent group $\Gamma'=\Gamma/B$ and the epimorphism  $\phi':\Gamma'\to\pi$.
induced by $\phi$. In view of the principal fiber bundle $B\Gamma\to B\Gamma'$ with the fiber a torus, it suffices to show that $\phi'$ can be realized as a fiber bundle of nilmanifolds.
We consider the commutative diagram
$$
\begin{CD}
\Gamma'@>>> \Gamma'/A\\
@V\phi'VV @V\bar\phi VV\\
\pi @>>> \pi/\bar A.\\
\end{CD}
$$
Note that $\Gamma'/A=\Gamma/Z(\Gamma)$ and, hence, $\bar\phi$ is an epimorphism with $\nil(\Gamma'/A)+\nil(\pi/\bar A)<s+t.$
The homomorphism $\phi'$ factors through the pull-back $\Lambda=\bar\phi^*\pi$, $\phi'=\bar\phi'\circ\xi$ with respect to $\bar\phi$.
Since $\phi'|_A:A\to\bar A$ is an embedding of a finite index subgroup, the homomoprhism $\xi:\Gamma'\to\Lambda$
 in the commutative diagram of short exact sequences generated by $\phi'$ and the pull-back
$$
\begin{CD}
1 @>>> A @>>> \Gamma' @>>> \Gamma/Z(\Gamma) @>>>1\\
@. @VV\subset V @V\xi V\subset V @VV= V @.\\
1 @>>>\bar A @>>> \Lambda @>>> \Gamma/Z(\Gamma) @>>>1\\
@. @VV= V @V\bar\phi'VV @V\bar\phi VV @.\\
1 @>>>\bar A @>>> \pi @>>> \pi/\bar A @>>>1\\
\end{CD}
$$
is an embedding of a finite index subgroup. 

By induction assumption  applied to $\bar\phi$
there is a fiber bundle between nilmanifolds
$$\bar f:B(\Gamma/Z(\Gamma)) \to B(\pi/\bar A)$$ with the fiber a nilmanifold $F$.
Consider the pull-back diagram
$$
\begin{CD}
  B\Lambda @>p'>> B(\Gamma/Z(\Gamma)) \\
 @V\bar f'VV @V\bar f VV \\
 B\pi @>p>> B(\pi/\bar A).\\
\end{CD}
$$
Let $q:M\to B\Lambda$ be a covering that corresponds to the subgroup $\xi(\Gamma')\subset\Lambda=\pi_1(B\Lambda)$.
Since $\xi(\Gamma')$ is of finite index, $M$ is a closed aspherical manifold with $\pi_1(M)=\Gamma'$. Thus, the composition $f'=\bar f'\circ q:M\to B\pi$ realizes the homomorphism $\phi'$ as a fiber bundle.
The fiber  $F'$ of  $f'$ is homeomorphic to the total space of a fiber bundle  $F'\to F$ with a finite fiber. The homotopy exact sequence of the fibration
$F'\to B\Gamma'\stackrel{f'}\to B\pi$ and the fact that $\phi'$ is surjective imply that $F'$ is connected. Hence, $F'$ is a nilmanifold.
\end{proof}

\begin{lemm}\label{compact support} 
For every locally trivial bundle of closed aspherical manifolds  $f:M^{m}\rightarrow N^{n}$ with compact connected fiber $F$ 
the induced homomorphism  $$f^{*}:H^{n}(N;\Z\pi) \rightarrow H^{n}(M;\Z\pi)$$ is nonzero where $\pi=\pi_1(N)$.
\end{lemm}

\begin{proof}
 Consider the pull-back diagram:

\tikzset{node distance=2cm, auto}

\begin{tikzcd}[scale = 2]
\Wi{M} \arrow[r, "p_{2}", dashed] \arrow["p_{M}",dr] &
f^{*}{\Wi{N}} \arrow[d, "p_{1}"] \arrow[r,"\bar{f}"] & \Wi{N} \arrow[d, "p_{N}"]\\
& M \arrow[r,"f"] & N
\end{tikzcd}
\\  where $p_{M}$ and $p_{N}$ are he universal covering maps.
Since $f^{*}(\Wi{N})$ is a covering of $M$ and $\Wi{M}$ is the universal covering,  $p_M$ factors through covering maps $p_{2}$ and $p_1$.

We recall~\cite[Lemma 7.4]{Br} that for every left $\pi$-module $M$ there is a natural isomorphism of right modules $$\Theta:Hom_\pi(M,\Z\pi)\to Hom_c(M,\Z)$$ defined by the formula
 $F\mapsto f_1$ where $F:M\to\Z\pi$, $F(m)=\sum_{\gamma\in\pi} f_\gamma(m)\gamma$.
We consider the following commutative diagrams of cochain groups:

\begin{tikzcd}[scale = 2]
Hom_{\Gamma}(C_{*}(\Wi{M}),\Z\pi) \arrow[dr, "\Psi", dashed]\\
Hom_{\pi}(C_{*}(X),\Z\pi)  \arrow[r,"\Theta"] \arrow[u, "p_2^*"] & Hom_{c}(C_{*}(X),\Z)\\
Hom_{\pi}(C_{*}(\Wi{N}),\Z\pi) \arrow[r,"\Theta"] \arrow[u, "(\bar{f})^{\ast}"] & Hom_{c}(C_{*}(\Wi{N}),\Z) \arrow[u, "(\bar{f}_c)^{\ast}"]
\end{tikzcd}
\\ 
where $X=f^*\Wi N$, $\Psi=\Theta\circ\Phi$, and  $$\Phi:Hom_{\Gamma}(C_{*}(\Wi{M}),\Z\pi) \to Hom_{\pi}(C_{*}(X),\Z\pi)$$ is defined as follows. 
For each simplex $\sigma$ in $X$ we fix a lift $\tilde\sigma$  in $\Wi{M}$. 
We define $$\Phi(F)(\sigma)=F(\Wi{\sigma})$$ for $F\in Hom_\Gamma(C_*(\Wi M),\Z\pi)$ and simplex $\sigma$ in $X$ . 

{\em Claim 1.} $\Phi(F)\in Hom_\pi(C_*(X),\Z\pi)$.

Proof. Since the fiber $F$ is connected, it follows that $f$ induces an epimorphism of the fundamental groups.
Let $g\in\pi$ and let $f_*(\bar g)=g$. Then  $\Wi{g\sigma}$ and $\bar g\tilde\sigma$ both cover $g\sigma$. Hence there is $\gamma\in Ker f_*$ such that $\gamma(\Wi{g\sigma})=\bar g\tilde\sigma$. Then $\Phi$ is $\pi$-equivariant:
$$\Phi(F)(g\sigma)=F(\Wi{g\sigma})=F(\gamma\Wi{g\sigma})=F(\bar g\tilde\sigma)=gF(\tilde\sigma)=g\Phi(F)(\sigma).$$

{\em Claim 2.} $\Phi\circ p_2^* =1$.

Proof. $\Phi\circ p_2^*(F')(\sigma)=\Phi(p_2^*(F'))(\sigma)=(p_2^*F')(\tilde\sigma)=F'(p_2(\tilde\sigma))=F'(\sigma)$.

Since $\Wi{N}$ is contractible, the bottom row in the diagram above gives us an isomorphism of the cohomology groups $\Theta^*:H^*(N;\Z\pi)\to H^*_c(\Wi N;\Z)$.
The commutative diagram on the cochain level produces the commutative diagram for  cohomology:

\begin{tikzcd}
H^{n}(M;\Z\pi) \arrow[r, "\Psi^{*}"]
& H_{c}^{n}(X;\Z) \\
H^{n}(N;\Z\pi) \arrow[r, "\Theta^{*}" ] \arrow[u,"f_c^{*}"]
 & H_{c}^{n}(\Wi{N}; \Z)=\Z \arrow[u, "(\bar{f})^{*}"].
\end{tikzcd}

Since $\Wi{N}$ is contractible, $X=f^{*}(\Wi{N})$ is a trivial fiber bundle with the fiber $F$. 
Since $X\cong\Wi{N}\times F$ admits a proper retraction onto $\Wi{N}$,  we obtain that $(\bar{f_c})^{*}$ is a monomorphism.
Hence, the homomorphism $f^{*}:\Z=H^n(N;\Z\pi)\to H^n(M;\Z\pi)$ is not trivial.
\end{proof}

\begin{thm}
For a homomorphism $\phi:\Gamma\to\pi$ of finitely generated torsion free nilpotent groups $\cat\phi=\cd(\phi)$.
\end{thm}
\begin{proof}
In view of Theorem~\ref{cat of mono} and Theorem~\ref{cd of mono}, it suffices to prove the theorem when $\phi$ is surjective. By Lemma~\ref{nilbundle} there is a fiber bundle 
of nilmanifolds $M^m=B\Gamma\to B\pi=N^n$ with a compact fiber.
By Lemma~\ref{compact support}, $\cd(\phi)=n$. The inequalities $\cd(\phi)\le\cat\phi\le\dim N=n$ complete the proof.
\end{proof}

\section{Counterexample}

Answering a question of M. Grant,
T. Goodwillie gave an example of an epimorphism $\phi:G\to\Gamma$ satisfying the inequality $\cd(\phi)<\min\{\cd(G),\cd(\Gamma)\}.$ We present here a slightly modified version
of his example. Let $\Gamma=\Z^2$ and $\phi:G\to\Gamma$ be the epimorphism defined by the extension of $\Gamma$ by the abelian group $I(\Gamma)^2=I(\Gamma)\otimes_\Z I(\Gamma)$ that corresponds to the square of the Berstein-Schwarz cohomology class $(\beta_\Gamma)^2\in H^2(\Gamma,I(\Gamma)^2)$. Then $\phi^*(\beta_\Gamma)^2=0$. Hence, $\cd(\phi)<2$. 
If $\cat\phi\le 1$, then the induced map $\hat\phi:BG\to B\Gamma$ can be lifted to the Ganea's space $G_1(B\Gamma)$ which is homotopy equivalent to 1-dimensional complex  
$\Sigma\Gamma$,
the suspension of $\Gamma$.
~\cite{CLOT}. Then the epimorphism $\phi$ factors through a free group $F$ via epimorphisms $G\to F\to \Gamma$. Hence  $rank F\ge 2$. Therefore, $G$ contains a free group
of rank $\ge 2$. We note that $G$ is amenable as an abelian-by-abelian extension and hence  $G$ cannot contain $F$.
Therefore, $\cat\phi\ge 2$. Thus, $\phi$ is an infinitely generated counterexample to Conjecture~\ref{con}.

\begin{rmk} 
We note there are simpler examples of epimorphisms $\phi:G\to\Gamma$ satisfying  $$\cd(\phi)<\min\{\cd(G),\cd(\Gamma)\}.$$ Namely, the homomorphism $\phi:F_n\ast\Z^n\to\Z^n$
defined by the abelianization of the first factor and by the projection onto $\Z^n\to\Z\subset\Z^n$ has $\cd(\phi)=1$ and $\cd(F^n\ast\Z^n)=\cd\Z^n=n$. Here $F_n$ denotes the free group on $n$ generators.
\end{rmk}
Let $\S$ denote the sphere spectrum.
We recall that every stably parallelizable manifold is $\S$-orientable~\cite{Sw}.
The $\S$-cohomology groups of $X$  are exactly the stable cohomotopy groups $\pi^*_s(X)$.
Then Lemma 3.5 of~\cite{Ru} in the case of the spectrum $\S$ can be stated as follows:
\begin{lemm}\label{rudyak}
Suppose that $f:W\to M$ is a map of degree one  between closed stably paralleliable manifolds. Then the induced map $f^*:\pi^*_s(M)\to\pi_s^*(W)$ is injective.
\end{lemm}

We note that the natural map $[X,S^n]\to\pi^n_s(X)$ is a bijection when $\dim X\le 2n-2$~\cite{Hu}.

\subsection{Bolotov's example} Answering a question of Gromov~\cite{Gro1},
D. Bolotov constructed~\cite{Bo} a closed 4-manifold $M$ with the fundamental group $\pi=\Z\ast \Z^3$ whose cohomological dimension $\cd(\pi)=3$ such that 
a classifying map $u_M:M\to B\pi$ cannot be deformed to the 2-skeleton $B\pi^{(2)}$. His manifold is defined as the pull-back $M=g^*(S^3\times S^1)$
of the $S^1$-bundle $h\times 1:S^3\times S^1\to S^2\times S^1$, where $h$ is the Hopf fiber bundle, with respect to the collapsing map $g:N=(S^2\times S^1)\# T^3\to S^2\times S^1$.
Here $T^3$ is a 3-dimensional torus. The pull-back bundle is denoted by $p:M\to N$.
\begin{prop} \label{bolotov}
Bolotov's manifold has the following properties:

(a) The map $p$ induces an isomorphism of the fundamental groups.

(b) The homomorphism $u_M^*: H^3(B\pi;A)\to H^3(M;A)$ is trivial for any $\pi$-module $A$.

(c) The through map
$
M\stackrel{u_M}\rightarrow B\pi=S^1\vee T^3\stackrel{q_2}\to T^3\stackrel{q_1}\rightarrow S^3
$
is essential where $q_2$ collapses $S^1$ to the wedge point, and $q_1$ is a map of degree one.

(d) The manifold $M$ is stably parallelizable.
\end{prop}
\begin{proof}
(a) This is straightforward (\cite{Bo}). Thus, the classifying map $u_M$ factors through $p:M\to N$,  $u_M=q_3q_4p$ where $q_4:N\to (S^2\times S^1)\vee T^3 $ collapses the connected sum to the wedge and $q_3:(S^2\times S^1)\vee T^3\to S^1\vee T^3$ projects $S^2\times S^1$ onto the factor $S^1$. 

 (b) Let $a\in H^3(B\pi;A)$. By the Poincare Duality for local coefficients there is 1-dimensional class $\bar\beta$ in $M$ such that $ u_M^*(a)\cup\bar\beta\ne 0$. In view of universality of the Berstein-Schwarz class $\beta_\pi$ and the fact that $u_M$ induces an isomorphism of 1-cohomology,
we may assume that $\bar\beta=u_M^*(\beta_\pi)$. Then $0\ne u_M^*(a)\cup\bar\beta=u_M^*(a\cup\beta_\pi)=0$. The last equality is due to dimensional reason.

(c)  This was proven in~\cite{Bo}. Here we present a simplified proof. Let $q=q_1q_2q_3q_4:N\to S^3$:
$$
N=(S^2\times S^1)\# T^3\stackrel{q_4}\to (S^2\times S^1)\vee T^3\stackrel{q_3}\to S^1\vee T^3\stackrel{q_2}\to T^3\stackrel{q_1}\to S^3.
$$
Let $a$ be a generator in $H^3(K(\Z,3))$ and let $a_0$ be its restriction to $S^3=K(\Z,3)^{(4)}$.
Part (b) implies that $p^*(\alpha_0)= 0$ where $\alpha_0=q^*(a_0)$. The exact sequence of pair $(C_p,N)$ implies  $(j^*)^{-1}(\alpha_0)\ne\emptyset$ 
where $j:N\to C_p$ is the inclusion of the codomain of $p$ into the mapping 
cone. It suffices to show that $q:N\to S^3$ does not extend to $C_p$. Every extension $\psi'$ of $q$ to the 4-skeleton $C_p^{(4)}$ can be extended to a map $\psi: C_p\to K(\Z,3)$
defining an element   $\alpha\in (j^*)^{-1}(\alpha_0)$. We show that such $\psi$ cannot be deformed to $S^3\subset K(\Z,3)$.
Here we assume  that $K(\Z,3)^{(5)}=S^3\cup_\nu D^5$. 

For a cohomology class $x$ we denote by $\bar x$ its mod 2 reduction. Since $q$ is a map of degree one, $\bar\alpha_0\ne 0$.
We recall that the obstruction to retraction of  $S^3\cup_\nu D^5$ to $S^3$ is the Steenrod square $Sq^2\bar a$. By the naturality of primary obstructions the obstruction to deform the map
$\psi:C_p\to K(\Z,3)$ to $S^3$  is $\psi^*(Sq^2\bar a)=Sq^2\bar\alpha$. We show that $Sq^2\bar\alpha\ne 0$.
Let $E$ be the disk bundle associated with the circle budle $p$. For the use of cohomology we can identify $N$ with $E$. Then $M=\partial E$ and the pair $(E,M)$ can be identified with the mapping cone $C_p$.
In view of the Thom isomorphism there is $x\in H^1(N;\Z_2)$ such that $x\cup u=\bar\alpha$ where $u$ is the mod 2 Thom class. Since $j^*(\bar\alpha)=\bar\alpha_0$ the naturality of the cup product
$$
\begin{CD}
H^3(N,\partial N;\Z_2)\times H^2(N,\partial N;\Z_2) @>\cup>> H^5(N,\partial N;\Z_2)\\
@Vj^*\times 1VV @V1VV\\
H^3(N;\Z_2)\times H^2(N,\partial N;\Z_2) @>\cup>> H^5(N,\partial N;\Z_2)\\
\end{CD}
$$
implies that $\bar\alpha_0\cup u=\bar\alpha\cup u$.
Then
$$0\ne\bar\alpha_0\cup u=\bar\alpha\cup u=x\cup u\cup u=x\cup Sq^2 u=Sq^2(x\cup u)=Sq^2(\bar\alpha).$$
Here we use the equality $Sg^2u=u\cup u$ and the Cartan formula for Steenrod squares.

(d) The manifold $M$ is stably parallelizable as the total space of an orientable $S^1$-bundle over a stably parallelizable manifold $N$.
\end{proof}

We recall that for every closed manifold $M$ there is a hyperbolization $f:W\to M$ which a map of degree one of a closed aspherical manifold which is surjective on the fundamental groups. Moreover, the map $f$ induces an isomorphism between stable tangent bundles of $M$ and $W$~\cite{DJ},\cite{CD},\cite{Gro2}.

\begin{thm}
There is a map of a closed aspherical 4-manifold $g:W\to T^3$ onto a 3-torus that induces an epimorphism of the fundamental groups $g_\#:\pi_1(W)\to\Z^3$ such that  $\cat g_\#=3$ and $\cd(g_\#)<3$.
\end{thm}
\begin{proof}
Let $M$ be the Bolotov's example and let $f:W\to M$ be the  hyperbolization of $M$.
Let $g=q_2\circ u_M\circ f$. Since Bolotov's example is stably parallelizable and the hyperbolization is a tangential map,
by Lemma~\ref{rudyak} we obtain that $f^*:\pi_s^3(M)\to\pi_s^3(W)$ is injective.
By dimensional reason $[M,S^3]=\pi_s^3(M)$ and $[W,S^3]=\pi_s^3(W)$. Thus, since by Proposition~\ref{bolotov} (c) the map $q_1q_2u_M:M\to S^3$ is essential, 
the map $q_1\circ g:W\to S^3$ is essential as well. Hence, the map $g$ cannot be deformed to the 2-skeleton. Therefore,
$\cat g>2$. Clearly, $\cat g=\cat g_\#=3$.

By Proposition~\ref{bolotov} (a) the homomorphism  $(q_1\circ u_M)^*$ is trivial on 3-dimensional cohomology. Hence, so is $g^*$. This means that $\cd(g_\#)<3$.
\end{proof}

\section*{Acknowledgments}
The first author was supported by the Simons Foundation Grant.

\

\end{document}